\documentclass[a4paper,12pt]{article}

\usepackage{amsfonts}
\usepackage{amscd,color}
\usepackage{amsmath,amsfonts,amssymb,amscd}
\usepackage{indentfirst,graphicx,epsfig}
\usepackage{graphicx}
\input{epsf}
\usepackage{graphicx}
\usepackage{epstopdf}
\usepackage{caption}

\usepackage{fullpage}
\usepackage{epsfig}
\usepackage{latexsym}
\usepackage{caption}
\usepackage{amsmath,amsthm}
\usepackage{amsfonts}
\usepackage{amssymb}
\usepackage{graphicx}
\usepackage{graphics}
\usepackage{bm}
\usepackage{color}
\usepackage{subfigure}
\graphicspath{{\figs}}

\makeatletter

\newcommand{\Rmnum}[1]{\expandafter\@slowromancap\romannumeral #1@}
\makeatother
\makeatletter
\let\@fnsymbol\@arabic
\makeatother

\begin{document}

\newtheorem{theorem}{Theorem}[section]
\newtheorem{observation}[theorem]{Observation}
\newtheorem{corollary}[theorem]{Corollary}
\newtheorem{algorithm}[theorem]{Algorithm}
\newtheorem{problem}[theorem]{Problem}
\newtheorem{question}[theorem]{Question}
\newtheorem{lemma}[theorem]{Lemma}
\newtheorem{proposition}[theorem]{Proposition}

\newtheorem{definition}[theorem]{Definition}
\newtheorem{guess}[theorem]{Conjecture}
\newtheorem{claim}[theorem]{Claim}
\newtheorem{example}[theorem]{Example}
\newtheorem{remark}[theorem]{Remark}

\makeatletter
  \newcommand\figcaption{\def\@captype{figure}\caption}
  \newcommand\tabcaption{\def\@captype{table}\caption}
\makeatother

\newtheorem{acknowledgement}[theorem]{Acknowledgement}
\newtheorem{axiom}[theorem]{Axiom}
\newtheorem{case}[theorem]{Case}
\newtheorem{conclusion}[theorem]{Conclusion}
\newtheorem{condition}[theorem]{Condition}
\newtheorem{conjecture}[theorem]{Conjecture}
\newtheorem{criterion}[theorem]{Criterion}
\newtheorem{exercise}[theorem]{Exercise}
\newtheorem{notation}[theorem]{Notation}

\newtheorem{solution}[theorem]{Solution}
\newtheorem{summary}[theorem]{Summary}
\newtheorem{fact}[theorem]{Fact}

\newcommand{\pp}{{\it p.}}
\newcommand{\de}{\em}

\newcommand{\mad}{\rm mad}

\newcommand*{\QEDA}{\hfill\ensuremath{\blacksquare}}  
\newcommand*{\QEDB}{\hfill\ensuremath{\square}}  

\newcommand{\qf}{Q({\cal F},s)}
\newcommand{\qff}{Q({\cal F}',s)}
\newcommand{\qfff}{Q({\cal F}'',s)}
\newcommand{\f}{{\cal F}}
\newcommand{\ff}{{\cal F}'}
\newcommand{\fff}{{\cal F}''}
\newcommand{\fs}{{\cal F},s}

\newcommand{\g}{\gamma}
\newcommand{\wrt}{with respect to }

\def\C#1{|#1|}
\def\E#1{|E(#1)|}
\def\V#1{|V(#1)|}
\def\cB{{\mathcal B}}
\def\cD{{\mathcal D}}
\def\cF{{\mathcal F}}
\def\cI{{\mathcal I}}
\def\cP{{\mathcal P}}
\def\cT{{\mathcal T}}
\def\NN{{\mathbb N}}
\def\VEC#1#2#3{#1_{#2},\dots,#1_{#3}}
\def\VECOP#1#2#3#4{#1_{#2}#4\dots#4#1_{#3}}
\def\vp{\varphi}
\def\ve{\varepsilon}
\def\FR{\frac}
\def\st{\colon\,}
\def\esub{\subseteq}
\def\nul{\varnothing}
\def\SE#1#2#3{\sum_{#1=#2}^{#3}}

\def\Phi{\vp}
\def\iarb{\Upsilon}
\def\ipac{\nu}

\title{Monochromatic $k$-edge-connection colorings of graphs\footnote{Supported by NSFC No.11871034, 11531011 and NSFQH No.2017-ZJ-790.}}

\renewcommand{\thefootnote}{\arabic{footnote}}

\author{\small Ping Li$^1$, ~ Xueliang Li$^{1,2}$\\
\small $^1$Center for Combinatorics and LPMC, Nankai University\\
\small Tianjin 300071, China\\
\small $^2$School of Mathematics and Statistics, Qinghai Normal University\\
\small Xining, Qinghai 810008, China\\
\small qdli\underline{ }ping@163.com, ~ lxl@nankai.edu.cn\\
}

\date{}
\maketitle


\begin{abstract}
A path in an edge-colored graph $G$ is called monochromatic if any two edges on the path have the same color. For $k\geq 2$, an edge-colored graph $G$
is said to be monochromatic $k$-edge-connected if every two distinct vertices of $G$ are connected by at least $k$ edge-disjoint monochromatic paths, and
$G$ is said to be uniformly monochromatic $k$-edge-connected if every two distinct vertices are connected by at least $k$ edge-disjoint monochromatic paths such
that all edges of these $k$ paths colored with a same color. We use $mc_k(G)$ and $umc_k(G)$ to denote the maximum number of colors that ensures $G$
to be monochromatic $k$-edge-connected and, respectively, $G$ to be  uniformly monochromatic $k$-edge-connected.
In this paper, we first conjecture that for any $k$-edge-connected graph $G$, $mc_k(G)=e(G)-e(H)+\lfloor\frac{k}{2}\rfloor$, where $H$ is a minimum $k$-edge-connected spanning subgraph of $G$. We verify the conjecture for $k=2$. We also prove the conjecture for
$G=K_{k+1}$ when $k\geq4$ is even, and for
$G=K_{k,n}$ when $k\geq4$ is even, or when $k=3$ and $n\geq k$.
When $G$ is a minimal $k$-edge-connected graph, we give an upper bound
of $mc_k(G)$, i.e., $mc_k(G)\leq k-1$, and
$mc_k(G)\leq \lfloor\frac{k}{2}\rfloor$ when $G=K_{k,n}$.
For the uniformly monochromatic $k$-edge-connectivity, we prove that for all $k$, $umc_k(G)=e(G)-e(H)+1$, where $H$ is a minimum $k$-edge-connected spanning subgraph of $G$.\\
[2mm] {\bf Keywords:} edge-coloring, monochromatic path, edge-connectivity, monochromatic $k$-edge connection number.\\
[2mm] {\bf AMS subject classification (2010)}: 05C15, 05C40.
\end{abstract}

\baselineskip16pt

\section{Introduction}

All graphs in this paper are simple and undirected. For a graph $G$, we use $V(G), E(G)$ to denote the vertex set and edge set of $G$, respectively, and $e(G)$
the number of edges of $G$. For all other terminology and notation not defined here we follow Bondy and Murty \cite{B}.

For a natural number $r$, we use $[r]$ to denote the set $\{1,2,\cdots, r\}$ of integers. Let $\Gamma: E(G)\rightarrow [r]$ be an edge-coloring of $G$ that allows a
same color to be assigned to adjacent edges. For two vertices $u$ and $v$ of $G$, a {\em monochromatic uv-path} is a $uv$-path of $G$ whose edges are
colored with a same color, and $G$ is {\em monochromatic connected } if any two distinct vertices of $G$ are connected by a monochromatic path.
An edge-coloring $\Gamma$ of $G$ is a {\em monochromatic connection coloring (MC-coloring) }if it makes $G$ monochromatic connected. The {\em monochromatic connection number}
of a connected graph $G$, denoted by $mc(G)$, is the maximum number of colors that are needed in order to make $G$ monochromatic connected. An {\em extremal MC-coloring}
of $G$ is an $MC$-coloring that uses $mc(G)$ colors.

The notion monochromatic connection coloring was introduced by Caro and Yuster in \cite{CY}. Many results have been obtained; see \cite{CLW, GLQZ, JLW, MWYY}. For more
knowledge on the monochromatic connections of graphs we refer to a survey paper \cite{LW2}.
Gonzlez-Moreno, Guevara, and Montellano-Ballesteros in \cite{MGM} generalized the above concept to digraphs. Now we introduce the concept of {\em monochromatic $k$-edge-connectivity} of graphs.
An edge-colored graph $G$ is {\em monochromatic $k$-edge-connected} if every two distinct vertices are connected by at least $k$ edge-disjoint monochromatic paths
(allow some of the paths to have different colors). An edge-coloring $\Gamma$ of $G$ is a {\em monochromatic $k$-edge-connection coloring ($MC_k$-coloring) }if it makes $G$
monochromatic $k$-edge-connected. The {\em monochromatic $k$-edge-connection number}, denoted by $mc_k(G)$, of a connected graph $G$ is the maximum number of colors that are needed in order to make $G$
monochromatic $k$-edge-connected. Since we can color all the edges of a $k$-edge-connected graph by distinct colors, $mc_k(G)$ is well-defined.
An {\em extremal $MC_k$-coloring} of $G$ is an $MC_k$-coloring that uses $mc_k(G)$ colors.

In an edge-colored graph $G$, we say that a subgraph $H$ of $G$ is induced by color $i$ if $H$ is induced by all the edges with a same color $i$ of $G$. If a color $i$ only color one edge of $E(G)$, then we call the color $i$ is a {\em trivial color}, and the edge is a {\em trivial edge}; otherwise, we call the colors (edges) {\em non-trivial}. We call an extremal $MC_k$-coloring  a {\em good $MC_k$-coloring} of $G$ if the coloring has the maximum number of trivial edges.

Suppose that $X$ is a proper vertex subset of $G$. We use $E(X)$ to denote the set of edges with both ends in $X$. For a graph $G$ and $X\subset V(G)$,
to shrink $X$ is to delete all edges in $E(X)$ and then merge the vertices of $X$ into a single vertex.  A partition of a vertex set $V$ is to divide $V$
into some mutual disjoint nonempty sets. Suppose $\cP=\{V_1,\cdots,V_s\}$ is a partition of $V(G)$. Then $G/\cP$ is a graph obtained from $G$ by shrinking every $V_i$ into a single vertex.

An edge $e$ of a $k$-edge-connected graph $G$ is {\em deletable} if $G\backslash e$ is also a $k$-edge-connected graph.
A $k$-edge-connected graph $G$ is {\em minimally $k$-edge-connected} if none of its edges is deletable. A {\em minimal $k$-edge-connected spanning subgraph} of $G$
is a $k$-edge-connected spanning graph of $G$ that does not have any deletable edges. {\em A minimum $k$-edge-connected spanning subgraph} of $G$ is a minimal $k$-edge-connected spanning subgraph
of $G$ that has minimum number of edges. The next result was obtained by Mader.

\begin{theorem} [Mader~\cite{M}] \label{Mader}
 Let $G$ be a minimally $k$-edge-connected graph of order $n$. Then
 \begin{enumerate}
 \item $e(G)\leq k(n-1)$.
 \item every edge $e$ of $G$ is contained in a $k$-edge cut of $G$.
 \item $G$ has a vertex of degree $k$.
 \end{enumerate}

\end{theorem}

The following theorem was proved by Nash-Williams and Tutte independently.

\begin{theorem}[\cite{NW}~\cite{T}] \label{N-W-T}
 A graph $G$ has at least $k$ edge-disjoint spanning trees if and only if $e(G/\cP)\geq k(|G/\cP|-1)$ for any vertex partition $\cP$ of $V(G)$.
\end{theorem}

We denote $\psi(G)=min_{|\cP|\geq2}\frac{e(G/\cP)}{|G/\cP|-1}$, and $\Psi(G)=\lfloor\psi(G)\rfloor$. Then the Nash-Williams-Tutte theorem can be restated as follows.

\begin{theorem} \label{N-T}
 A graph $G$ has exactly $k$ edge-disjoint spanning trees if and only if $\Psi(G)=k$.
\end{theorem}

If $\Gamma$ is an extremal $MC_k$-coloring of $G$, then each color-induced subgraph is connected; otherwise we can recolor the edges of one of its components
by a fresh color, and then the new coloring is also an $MC_k$-coloring of $G$, but then the number of colors is increased by one, which contradicts that $\Gamma$ is extremal.

For the monochromatic $k$-edge-connection number of graphs, we conjecture that the following statement is true.

\begin{conjecture} \label{coj}
For a $k$-edge-connected graph $G$ with $k\geq 2$, $mc_k(G)=e(G)-e(H)+\lfloor\frac{k}{2}\rfloor$, where $H$ is a minimum $k$-edge-connected spanning subgraph of $G$.
\end{conjecture}

In Section 2, we will prove that the conjecture is true for $k=2$, and that it is also true for some special graph classes.
We also give a lower bound of $mc_k(G)$ for $2\leq k\leq \Psi(G)$, and an upper bound of $mc_k(G)$ for minimally $k$-edge-connected graphs
with $k\geq 2$.

The following lemma seems easy, but it is useful for some proofs in Section 2.

\begin{lemma} \label{c2}
Suppose that $G$ is a $2$-edge-connected graph and $H$ is a $2$-edge-connected subgraph of $G$. Let $S$ be subset of $E(G)$ whose ends are contained in
$V(H)$ such that $S\cap E(H)=\emptyset$. Then $G\backslash S$ is also a $2$-edge-connected graph.
\end{lemma}
\begin{proof}
We need to show that for any $u,v$ in $G\backslash S$ there are at least two edge-disjoint paths connecting them. From the condition, there are two edge-disjoint $uv$-path $P_1,P_2$ in $G$.
Suppose $a_1$ is the first vertex of $V(P_1)$ from $u$ to $v$ contained in $V(H)$, and $a_2$ is the first vertex of $V(P_2)$ from $u$ to $v$  contained in $V(H)$ (if $u\in V(H)$, then $u=a_1=a_2$);
suppose $b_1$ is the last vertex from $u$ to $v$ contained in $V(H)$, and $b_2$ is the last vertex of $V(P_2)$ from $u$ to $v$ contained in $V(H)$ (if $v\in V(H)$, then $v=b_1=b_2$).
Let $L_i=uP_ia_i$ and $L_{i+2}=b_iP_iv$, $i=1,2$. Because each of $L_i$ does not contain any edge of $S$ and $H$ is a $2$-edge-connected graph, we have that $H\cup \bigcup_{i\in [4]}L_i$
is also a $2$-edge-connected graph of $G\backslash S$. Therefore, there are two edge-disjoint $uv$-paths in $G\backslash S$.
\end{proof}

In Section 3, we introduce other version of monochromatic $k$-edge-connection of graphs, i.e., uniformly monochromatic $k$-edge-connection
of graphs, and get some results. For details we will state them there.

\section{Results on the monochromatic $k$-edge-connection number}

\begin{theorem}\label{T2}
Conjecture \ref{coj} is true when $G$ and $k$ satisfy one of the following conditions:
 \begin{enumerate}
\item $k=2$, i.e., $G$ is a $2$-edge-connected graph.
 \item $G=K_{k+1}$ where $k\geq 4$ is even.
 \item $G=K_{k,n}$ where $k\geq 4$ is even, and $k=3$ and $n\geq k$.
 \end{enumerate}
\end{theorem}

We restate the first result of Theorem \ref{T2} as follows.

\begin{theorem} \label{mc1}
Let $G$ be a $2$-edge-connected graph. Then $mc_2(G)=e(G)-e(H)+1$, where $H$ is a minimum $2$-edge-connected spanning subgraph of $G$.
\end{theorem}

The following is the proof of Theorem \ref{mc1}. For convenience, we abbreviate the term ``monochromatic path" as ``path" in the proof.

Let $\Gamma$ be a good $MC_2$-coloring of $G$.
Then we denote the set of non-trivial colors of $\Gamma$ by $[r]$, and denote $G_i$ as a subgraph induced by the color $i$; subject to above,
let $p(\Gamma)=\Sigma_{i\in [r]}p(G_i)$ be maximum, where $p(G_i)$ is the number of non-cut edges of $G_i$. It is obvious that each of these edges is
contained in some cycles of $G_i$.

\begin{claim} \label{3-1}
Each $G_i$ is either a $2$-edge-connected graph or a tree.
\end{claim}

\begin{proof}
Suppose that $G_i$ is neither a $2$-edge-connected graph nor a tree, i.e., $G_i$ contains both non-trivial blocks and cut edges.
Therefore we can choose a cut edge $e=uv\in E(G_i)$ such that $v$ belongs to a maximal $2$-edge-connected subgraph $B$ of $G_i$
(actually, $B$ is the union of some non-trivial blocks). Because $B$ is a $2$-edge-connected subgraph of $G_i$, each of its vertices belongs
to a cycle. Let $v$ be contained in a cycle $C$ of $B$ and $e'=vw$ be an edge of $C$. Because $e$ is a cut edge of $G_i$, there is just
one $uw$-path in $G_i$ (the $uw$-path is $P$). Therefore, there exists another $uw$-path $P'$, which is colored differently from $i$.

If $P'$ is a path colored by $j$, then we can obtain a new coloring $\Gamma'$ of $G$ from $\Gamma$ by recoloring all edges of $G_i-e'$ with $j$.
We first prove that $\Gamma'$ is an $MC_2$-coloring of $G$, i.e., we need to prove that for any two vertices $a, b$ of $V(G)$, there are at least two $ab$-paths under $\Gamma'$.
If at least one vertex of $a, b$ does not belong to $V(G_i)$, then the two $ab$-paths are colored differently from $i$. Because we just change the color $i$,
the two $ab$-paths are not affected; if both of $a, b$ belong to $V(G_i)$ and at least one of them does not belong to $V(B)$, then
we can choose a right $ab$-path such that it does not contain $e'$ (under $\Gamma$), and so there are at least two $ab$-paths under $\Gamma'$;
if both $a, b\in V(B)$, then the two $ab$-paths under $\Gamma$ (call them $L_1,L_2$) belong to $B$. If $e'$ is not an edge of any $L_1,L_2$,
then the two $ab$-paths are not affected. Otherwise, let $e'\in E(L_1)$, and then $L=L_1-e'+e+P'$ is a trial connecting $a,b$.
Because $E(L)\cap E(L_2)=\emptyset$, there are two $ab$-paths under $\Gamma'$.

According to the above, $\Gamma'$ is an $MC_2$-coloring of $G$. If $j\in [r]$ is a non-trivial color, then the number of colors has not changed,
but the number of trivial edges is increased by one, which contradicts that $\Gamma$ is good; otherwise, if $j$ is a trivial color, i.e.,
$uw$ is a trivial edge, then the new coloring $\Gamma'$ is a good $MC_2$-coloring (the number of colors and non-trivial edges have not changed),
but compared to $p(\Gamma)$, $p(\Gamma')$ is increased by one, which contradicts that $p(\Gamma)$ is maximum.
Therefore, we have proved that $G_i$ is either a $2$-edge-connected graph or a tree.

\end{proof}

By Claim \ref{3-1}, each $G_i$ is either a $2$-edge-connected graph or a tree. Suppose there are $h$ trees and $s=k-h$ $2$-edge-connected graphs.
W.l.o.g., suppose that $G_1,\cdots,G_s$ are $s$ $2$-edge-connected graphs and $G_{s+1}=T_1,\cdots,G_k=T_h$ are $h$ trees.
$G_i$ colored by $i$ and $F_j$ colored by $s+j$. For convenience, we also call the color of $F_j$ $j$ when there is no confusion.
\begin{claim} \label{3-2}
For each $G_i$ and $T_j$, let $e=uv\in E(G_i)$ and $e'=xy\in E(T_j)$. Then at most one of $u,v$ belongs to $V(T_j)$, and at most one of $x,y$ belongs to $V(G_i)$.
\end{claim}
\begin{proof}
We prove it by contradiction, i.e., suppose that there exist $G_i$ and $T_j$, and there exist $e=uv\in E(G_i)$ and $e'=xy\in E(T_j)$, such that either $u,v\in V(T_j)$ or $x,y\in V(G_i)$.

{\em Case 1:} Suppose $u,v\in V(T_j)$. Then we recolor $E(G_i)-e$ by $j$ and keep the color of $e$. We now prove that the new coloring (call it $\Gamma'$) is an extremal $MC_2$-coloring of $G$.

We denote the segment of $uT_jv$ by $L$. For any pair of vertices $a,b$ of $V(G)$, if at least one vertex does not belong to $V(G_i)$,
then the two $ab$-paths colored differently from $i$ under $\Gamma$. Because we just change the color $i$, the two $ab$-paths are not affected;
if $a,b\in V(G_i)$, because $G_i+L-e$ is also $2$-edge-connected, then there are two $ab$-paths (with the same color $j$) under $\Gamma'$. Therefore, $\Gamma'$ is an $MC_2$-coloring,
and because the number of colors are not changed, $\Gamma'$ is still an extremal $MC_2$-coloring. However, the number of non-trivial edges is increased ($e$ becomes a trivial edge),
which contradicts that $\Gamma$ is good.

{\em Case 2:} Suppose $x,y\in V(G_i)$. Then we recolor $E(T_j)-e'$ with $i$ and keep the color of $e'$. We now prove that the new coloring (call it $\Gamma'$) is an extremal $MC_2$-coloring of $G$.

For any vertices pair $a,b$ of $V(G)$, if at least one of $a,b$ does not belong to $V(T_j)$, then the two $ab$-paths colored differently from $j$.
Because we just change the color $j$, the two $ab$-paths are not affected; if $a,b\in V(T_j)$ and at leat one of $a,b$ does not belong $V(G_i)$, then there is just
one $ab$-path of $T_j$ and the other $ab$-paths colored differently from $i$ under $\Gamma$. Because $G_i\cup(T_j\backslash e')$ is connected and all of them colored by $i$ under $\Gamma'$,
there are two $ab$-paths under $\Gamma'$; if both $a,b\in V(G_i)$, then there are two $ab$-paths (with the same color $i$) under $\Gamma'$.
Above all, $\Gamma'$ is an $MC_2$-coloring of $G$. Because the number of colors are not changed, $\Gamma'$ is an extremal $MC_2$-coloring of $G$.
However, the number of non-trivial edges is increased ($e'$ becomes a trivial edge), which contradicts that $\Gamma$ is good.
\end{proof}

By Claim \ref{3-2}, for each edge $e'=xy$ of a $T_j$, the other $xy$-paths belong to some $T_q$; for each edge $e=uv$ of a $G_i$, the other $uv$-paths belong to some $G_l$.
\begin{claim} \label{3-3}
$h=0$, i.e., $G_i$ is a $2$-edge-connected graph for any $i\in [r]$.
\end{claim}
\begin{proof}
If $h\neq 0$, for an edge $e_1=v_1u_1\in E(T_1)$, because $P_1=e_1=v_1u_1$ is the only $v_1u_1$-path of $T_1$, there exists another $v_1u_1$-path $P_2$, then $|P_2|\geq2$ (because $G$ is simple),
and therefore the color of $P_2$ is non-trivial. By Claim \ref{3-2}, $P_2$ belongs to some $T_j$, w.l.o.g., suppose $j=2$. Then $e_1+T_2$ contains a unique cycle $C_1$.
Let $f_1=v_1u_2$ is a pendent edge of $P_2$, and $e_2=v_2u_2$ is the edge adjacent to $f_1$ in $P_2$. Then there exists a $v_2u_2$-path $P_3$ in $T_3$ and $e_2+T_3$ contains a unique cycle $C_2$.
Let $f_2=v_2u_3$ is a pendent edge of $P_3$, and $e_3=v_3u_3$ is the edge adjacent to $f_2$ in $P_3$. By repeating the process, we get a series of trees $T_1,T_2,\cdots$, paths $P_1,P_2,\cdots$ and edges $f_1=v_1u_2,f_2=v_2u_3,\cdots$, etc. Because there are at most $h<\infty$ trees, there is a $T_d$ which is the first tree appearing before (w.l.o.g., suppose $T_d=T_1$), and the $v_{d-1}u_{d-1}$-path $P_d$
is contained in $T_d=T_1$. Because there are at least two trees in this sequence, we have $d-1\geq2$. Then $f_1\in T_2,f_2\in T_3,\cdots,f_{d-2}\in T_{d-1}$; $P_2\in T_2,P_3\in T_3,\cdots,P_{d}\in T_{d}=T_1$, etc.  $T_1,\cdots,T_{d-1}$ are different trees. Let $H=\bigcup_{i\in [d-1]}T_i$.

In order to complete the proof, we need to construct a $2$-edge-connected subgraph $T$ of $H$, a connected graph $H'$, and an edge set $B$ of $H$ with $|B|=d-2$ below.

{\em Case 1}: $e_1\notin E(P_d)$.

We have already discussed above that $C_1=P_2+e_1, C_2=P_3+e_2, \cdots,C_{d-1}=P_d+e_{d-1}$. So, $T=C_1+C_2-e_2+C_3-e_3+\cdots+C_{d-1}-e_{d-1}=\bigcup_{i=1}^{d-1}C_i-B$ is a closed trail,
where $B=\bigcup_{i=2}^{d-1}e_i$, see Fig.\ref{fig3-3}(1). Therefore, $T$ is a $2$-edge-connected graph. Because the ends of every edge in $B$ belong to $V(T)$,
we have that $H'=\bigcup_{i\in [d-1]}T_i\backslash B$ is a connected graph.

{\em Case 2}: $e_1\in E(P_d)$.

Suppose $F_1,F_2$ are two small trees of $T_1\backslash e_1$ and let $v_1\in V(F_1)$, $u_1\in V(F_2)$. Then there is a $u_{d-1}v_1$-path $L_1$ and a $v_{d-1}u_1$-path $L_2$
(if $u_{d-1}$ connects $u_1$ and $v_{d-1}$ connects $v_1$, the situation is similar). Let
$$T'=v_1e_1u_1P_2u_2P_3u_3\cdots P_{d-2}u_{d-2}P_{d-1}u_{d-1}L_1v_1$$
and
$$T''=u_1P_2u_2P_3u_3\cdots P_{d-2}u_{d-2}P_{d-1}v_{d-1}L_2u_1.$$
It is obvious that both of $T'$ and $T''$ are closed trails and
$$T'\cap T''=u_1P_2u_2\cdots P_{d-2}u_{d-2}P_{d-1}v_{d-1}$$
is a trail. Therefore, $T=T'\cup T''=\bigcup_{i=1}^{d-1}C_i-B$ is a $2$-edge-connected graph, where $B=\bigcup_{i=1}^{d-2}f_i$, see Fig.\ref{fig3-3}(2).
Because the ends of each edge in $B$ belong to $V(T)$, $H'=\bigcup_{i\in [d-1]}T_i\backslash B$ is a connected graph.

\begin{figure}[h]
    \centering
    \includegraphics[width=350pt]{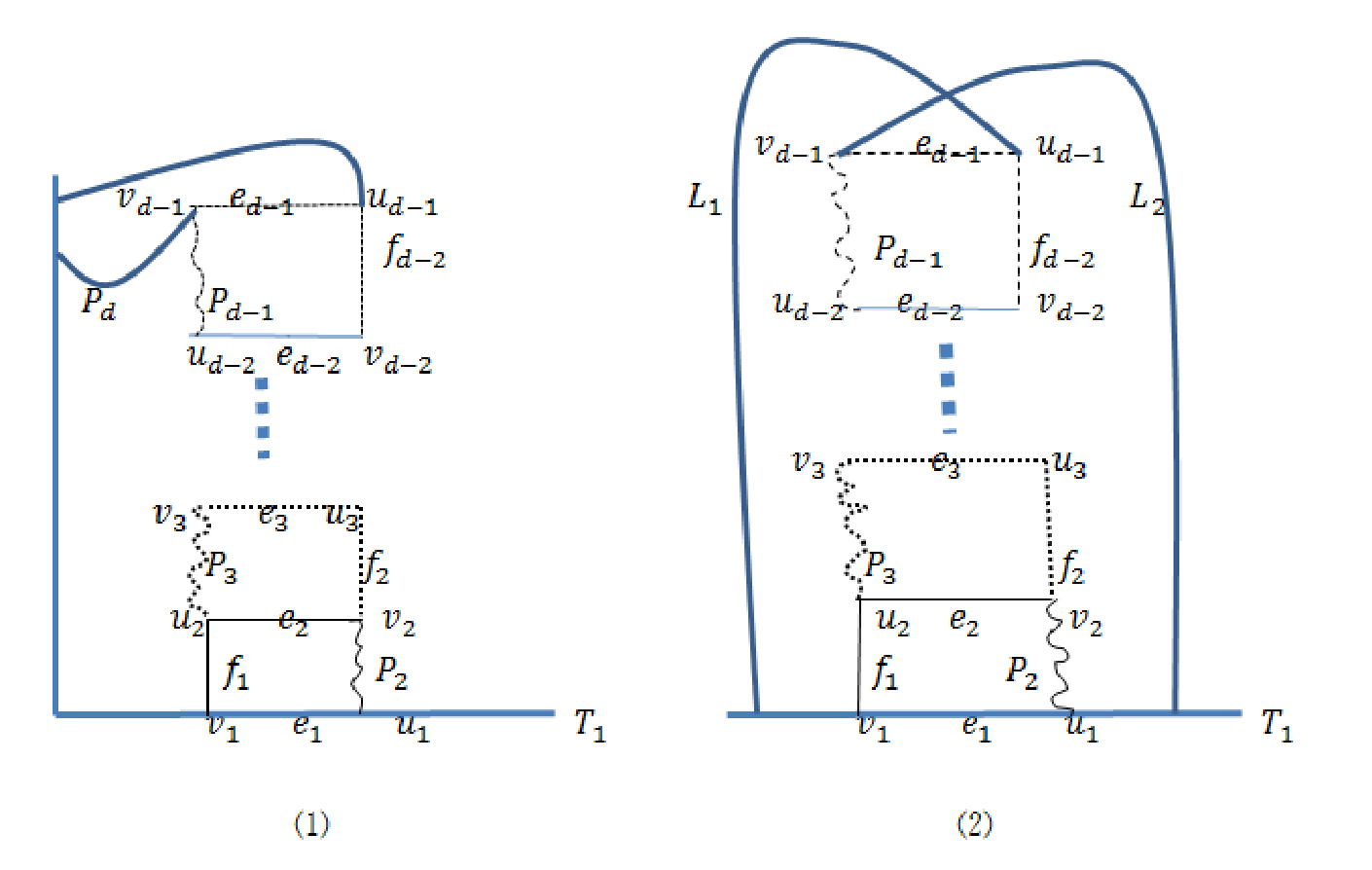}\\
    \caption{} \label{fig3-3}
\end{figure}

In above two cases, $T$ is a $2$-edge-connected subgraph of $H$, and $B$ is an edge set of $H$ with $|B|=d-2$.
We recolor each edges of $H-B$ by $1$ and recolor each edge of $B$ by  different new colors, denote the new coloring of $G$ by $\Gamma'$.
Then the total number of colors is not changed, but the number of trivial colors is increased by $|B|=d-2\geq1$. In order to complete the proof by contradiction,
we need to prove that $\Gamma'$ is an $MC_2$-coloring, i.e., we need to prove that for two distinct vertices $x,y$ of $G$, there are $2$ edge-disjoint $xy$-paths
under $\Gamma'$. There are three cases to discuss.

(\uppercase\expandafter{\romannumeral1})
 At least one of $x,y$ does not belong to $V(H)$. Then the two $xy$-paths do not belong to any $T_1,\cdots,T_{d-1}$.
 Because we just change the colors of $T_1,\cdots,T_{d-1}$, the two $xy$-paths are not affected from $\Gamma$ to $\Gamma'$.

(\uppercase\expandafter{\romannumeral2})
Both of $x,y$ belong to $V(H)$, but at least one of them does not belong to $V(T)$.

If there is just one $xy$-path in $H$ under $\Gamma$, then another $xy$-path will not be affected. Because $H'$ is connected, there are also two edge-disjoint $xy$-paths under $\Gamma'$.

If there are two $xy$-paths $L_1,L_2$ in $H$ under $\Gamma$. Suppose $a_i$ is the first vertex of $L_i$ contained in $V(T)$ from $x$ to $y$, and $b_i$ is the last vertex of $L_i$
contained in $V(T)$ from $x$ to $y$, $i=1,2$. Let $Q_i=xL_ia_i$ and $Q_{i+2}=b_iL_iy$, $i=1,2$. Because $T$ is a $2$-edge-connected graph,
$T\cup \bigcup_{i\in[4]}Q_i$ is also a $2$-edge-connected graph, i.e., there are two edge-disjoint $xy$-paths under $\Gamma'$.

(\uppercase\expandafter{\romannumeral3})
Both of $x,y$ belong to $V(T)$. Then because $T$ is a $2$-edge-connected graph, there are two edge-disjoint $xy$-path under $\Gamma'$.

\end{proof}

\begin{claim} \label{3-4}
$s=1$, i.e., all the non-trivial edges belong to $G_1$.
\end{claim}
\begin{proof}
The proof is done by contradiction.
If $s\geq2$, by Claim \ref{3-1}, each $G_i$ is a $2$-edge-connected graph. Thus, $V(G_1)\backslash V(G_2)\neq\emptyset$ and $V(G_2)\backslash V(G_1)\neq\emptyset$; for otherwise, w.l.o,g,  suppose $V(G_1)\subseteq V(G_2)$. Recoloring all the
edges of $G_1$ by different new colors, then the new coloring is an $MC_2$-coloring of $G$ but it has more colors than $\Gamma$, which contradicts that $\Gamma$ is extremal.

Let $a\in V(G_1)\backslash V(G_2)$ and $b\in V(G_2)\backslash V(G_1)$.
Suppose $G_a=\bigcup_{i\in c_a}G_i$ where $c_a=\{i:~a\in V(G_i)\}$. Let $t$ be the minimum integer such that $V(G_2)\subseteq V(\bigcup_{j\in [t]}G_{i_j})$ where $i_j\in c_a$. Then $t\leq |G_2|$. Recoloring the
edges of each $G_{i_j}$ by $i_1$, and recoloring the edges of $G_2$ by different new colors. Then the new coloring is an $MC_2$-coloring of $G$. Because $e(G_2)\geq |G_2|\geq t$, the number of colors is not decreased. However, the number of trivial colors
is increased, which contradicts that $\Gamma$ is good.
\end{proof}

\begin{claim} \label{3-41}
$G_1$ is a minimum $2$-edge-connected spanning subgraph of $G$.
\end{claim}
\begin{proof}
Because $s=1$ and $h=0$, there is just one non-trivial color (call it $1$). Then $G_1$ is a  $2$-edge-connected spanning subgraph of $G$; for otherwise,
there is a vertex $w\notin V(G_1)$, and then there is just one $uw$-path (which is a trivial path) for any $u\in V(G_1)$, a contradiction.

If $G_1$ is not minimum, we can choose a minimum $2$-edge-connected spanning subgraph $H$ of $G$ with $e(G_1)>e(H)$. Coloring each edge of $H$ by a same color
and coloring the other edges by trivial colors. Then the new coloring is an $MC_2$-coloring of $G$, but there are more colors than $\Gamma$, which contradicts
that $\Gamma$ is extremal.
\end{proof}

{\bf Proof of Theorem \ref{mc1}}: Actually, the theorem can be proved directly by Claims \ref{3-3}, \ref{3-4} and \ref{3-41}. Because $\Gamma$ is an
extremal $MC_2$-coloring of $G$, and the non-trivial color-inducted subgraph is just $G_1$, which is a minimum $2$-edge-connected spanning subgraph of $G$.
So, $mc_2(G)=e(G)-e(H)+1$ where $H$ is a minimum $2$-edge-connected spanning subgraph of $G$.
$\blacksquare$

We have proved that if $\Gamma$ is a coloring of $G$ in Theorem \ref{mc1}, then there is just one non-trivial color $1$ and $H=G_1$ is a minimum
$2$-edge-connected spanning subgraph of $G$. If $G$ has $t$ blocks, then $H$ also has $t$ blocks, and each block is a minimum $2$-edge-connected
spanning subgraph of the corresponding block of $G$. Furthermore, the number of edges of $H$ is greater than or equal to $n+t-1$ (equality holds if
each block of $H$ is a cycle). So, the following result is obvious.
\begin{corollary}
If $G$ is a $2$-edge-connected graph with $t$ blocks $B_1,\cdots,B_t$, then $mc_2(G)=\sum_{i\in [t]}mc_2(B_i)-t+1$, and $mc_2(G)\leq e(G)-n-t+2$.
\end{corollary}

A {\em cactus} is a connected graph where every edge lies in at most one cycle. If $G$ is a cactus without cut edges, then every edge lies in exactly one cycle.
It is obvious that $G$ will have cut edge when deleting any edge, and so $G$ is a minimal $2$-edge-connected graph. A minimal $k$-edge-connected graph is also the
minimum $k$-edge-connected spanning subgraph of itself, and this fact will not be declared again later.
\begin{corollary}
If $G$ is a cactus without cut edge, then $mc_2(G)=1$.
\end{corollary}

We have proved the first result of Theorem \ref{T2}. Next we will prove the remaining two results. Before this, we
give an upper bound of $mc_k(G)$ for $G$ being a minimal $k$-edge-connected graph. The following lemma is necessary for our later proof.
\begin{lemma} \label{total}
Let $G$ be a minimal $k$-edge-connected graph and $\Gamma$ be an extremal $MC_k$-coloring of $G$ (suppose $mc_k(G)=t$),
and let $G_i$ be the subgraph induced by the edges of color $i$, $1\leq i\leq t$. Then each $G_i$ is a spanning subgraph of $G$.
\end{lemma}
\begin{proof}
We prove it by contradiction. Suppose $G_i$ is not a spanning subgraph of $G$. Let $v\notin V(G_i)$. Then for any $u\neq v$,
none of the $k$ edge-disjoint monochromatic $uv$-paths is colored by $i$. Let $e$ be an edge colored by $i$. By Theorem \ref{Mader},
there exists an edge cut $C(G)$ such that $e\in C(G)$ and $|C(G)|=k$. Then $G\backslash C(G)$ has two components $M_1,M_2$ (in fact,
$C(G)$ is a bond of $G$). Let $v\in V(M_1)$ and some $w\in V(M_2)$. Then the $k$ edge-disjoint monochromatic $vw$-paths are retained
in $G\backslash e$. However, $C(G)\backslash e$ is an edge cut of $G\backslash e$ that separates $v$ and $w$, and $|C(G)\backslash e|=k-1$,
which contradicts that there are $k$ edge-disjoint monochromatic $vw$-paths in $G\backslash e$.
\end{proof}

\begin{theorem} \label{nk}
If $G$ is a minimal $k$-edge-connected graph with $k\geq2$, then $mc_k(G)\leq k-1$.
\end{theorem}
\begin{proof}
We prove it by contradiction. Suppose $mc_k(H)\geq k$. Let $\Gamma$ be an extremal $MC_k$-coloring of $G$.
Then by Lemma \ref{total}, there are at least $k$ edge-disjoint spanning subgraphs of $G$. Because there exists
a vertex of $G$ with degree $k$, there are exactly $k$ edge-disjoint spanning subgraphs of $G$, denoted by $G_1,\cdots,G_k$.
Because $G$ is a minimal $k$-edge-connected graph, by Theorem \ref{Mader}, $e(G)\leq k(n-1)$, which allows all of $G_1,\cdots,G_k$ to be spanning trees of $G$.

Because $k\geq 2$, there are at least two spanning trees $G_1, G_2$, and so $G_1\cup G_2$ is a $2$-edge-connected spanning subgraph of $G$.
Let $e=uv$ be an edge of $G_1$ and let $P_1$ be the $uv$-path of $G_2$. Suppose $e_1=uu_1$ and $e_2=vv_1$ are two terminal edges of $P_1$.
Let $P_2$ be the $uu_1$-path of $G_1$ and let $P_3$ be the $vv_1$-path of $G_1$.

{\em Case 1}: If one of $P_2$ and $P_3$ does not contain $e$, w.l.o.g., suppose $P_2$ does not contain $e$. Then $T=uP_2u_1P_1veu$ is a $2$-edge-connected graph
(in fact, $T$ is a closed trail, see Fig.\ref{minim}(1)). Because $u,u_1\in V(T)$, by Lemma \ref{c2}, $(G_1\cup G_2)\backslash e_1$ is a $2$-edge-connected subgraph
of $G$.

{\em Case 2}: If both $P_2$ and $P_3$ contain $e$,
then $T=uevP_2u_1P_1v_1P_3u$ is a $2$-edge-connected graph (in fact, $T$ is a closed trail, see Fig.\ref{minim}(2)). Because $u,u_1\in V(T)$, by Lemma \ref{c2},
$(G_1\cup G_2)\backslash e_1$ is a $2$-edge-connected subgraph of $G$.

\begin{figure}[h]
    \centering
    \includegraphics[width=350pt]{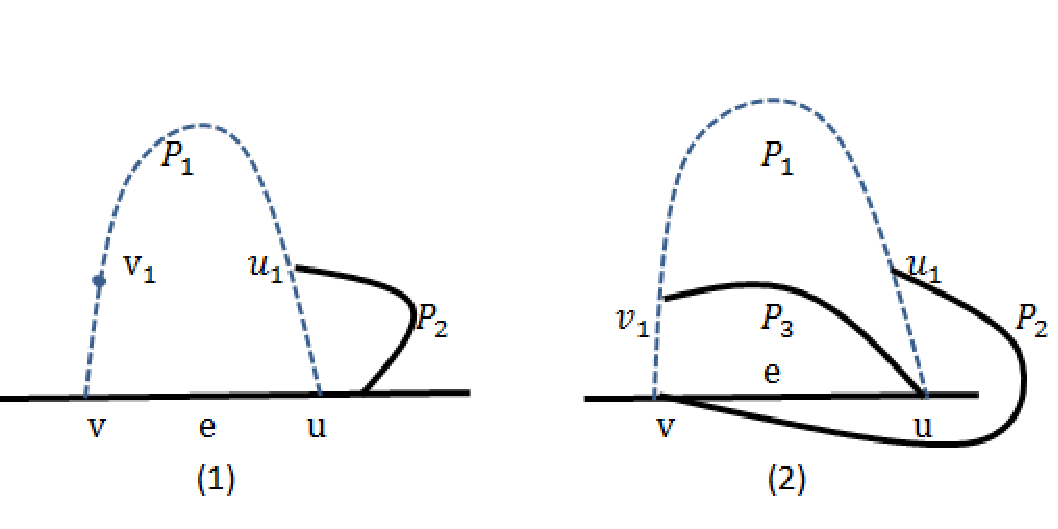}\\
    \caption{} \label{minim}
\end{figure}

The coloring $\Gamma'$ obtained from $\Gamma$ by assigning $1$ to the edges of $G_2\backslash e_1$ and assigning a new color to $e_1$. From above two cases,
$(G_1\cup G_2)\backslash e_1$ is a $2$-edge-connected spanning subgraph of $G$ and $G_3,\cdots,G_k$ are spanning subgraph of $G$. So, every two vertices are also connected by $k$
monochromatic paths and the number of colors is not changed, i.e., $\Gamma'$ is also an extremal $MC_k$-coloring of $G$. While $e$ is a single edge, that would contradict that each induced subgraph is spanning by Lemma \ref{total}.
\end{proof}
Before proving the second result of Theorem \ref{T2}, we introduce a well-known result.
\begin{fact} \label{signi}
$K_{2n+1}$ can be decomposed into $n$ edge-disjoint Hamiltonian cycles; $K_{2n+2}$ can be decomposed into $n$ edge-disjoint Hamiltonian cycles and a perfect matching.
\end{fact}

\begin{theorem}
$mc_{2n}(K_{2n+1})=n$ for $n\geq2$.
\end{theorem}
\begin{proof}
By Fact \ref{signi}, $K_{2n+1}$ can be decomposed into $n$ edge-disjoint Hamiltonian cycles $C_1,\cdots,C_n$.
Color each $C_i$ by $i\in [n]$, and then the coloring is an $MC_{2n}$-coloring of $K_{2n+1}$. So, $mc_{2n}(K_{2n+1})\geq n$.

We need to prove that $mc_{2n}(K_{2n+1})\leq n$ to complete our proof. The proof is done by contradiction.
Suppose $mc_{2n}(K_{2n+1})=t\geq n+1$. Let $\Gamma$ be an extremal $MC_{2n}$-coloring of $K_{2n+1}$ and let $G_i$ be the subgraph induced by all the edges with color $i$, $1\leq i\leq t$.

Because $K_{2n+1}$ is a minimal $2n$-edge-connected graph, by Lemma \ref{total} we have that each $G_i$ is a spanning subgraph of $G$. If $t\geq 2n$, then
$$n(2n+1)=e(K_{2n+1})=e(\bigcup_{i\in[t]}G_i)\geq 2tn\geq 4n^2,$$
which is a contradiction. Otherwise, if $t< 2n$, then not every $G_i$ is a spanning tree (for otherwise, every two vertices are just connected by $t<2n$ monochromatic paths). To ensure that
every two vertices are connected by at least $2n$ monochromatic paths, there are at least $2n-t$ $G_i$ that are $2$-edge-connected. Therefore, the number of edges of $\bigcup_{i\in[t]}G_i$ satisfies
$$e(\bigcup_{i\in[t]}G_i)\geq (2n+1)(2n-t)+2(t-n)\cdot 2n=t(2n-1)+2n\geq 2n^2+3n-1.$$
This contradicts that $\bigcup_{i\in[t]}G_i=K_{2n+1}$ and $e(K_{2n+1})=n(2n+1)$.
\end{proof}

Before prove the third result of Theorem \ref{T2}, we introduce another well-known result.
\begin{fact}\label{bip}
$K_{2n,2n}$ can be decomposed into $n$ Hamiltonian cycles and $K_{2n+1,2n+1}$ can be decomposed into $n$ Hamiltonian cycles and a perfect matching.
\end{fact}
\begin{theorem} \label{kkn}
If $n\geq k\geq 3$, then $mc_k(K_{k,n})\leq\lfloor\frac{k}{2}\rfloor$.
\end{theorem}
\begin{proof}
Let $\Gamma$ be an extremal $MC_k$-coloring with $t$ colors and let $G_i$ be the subgraph of $G$ induced by the edges with color $i$.
Because $K_{k,n}$ is a minimal $k$-edge-connected graph, by Lemma \ref{total} each $G_i$ is a spanning subgraph of $G$.
Let $A,B$ be the bipartition (independent sets) of $G$ with $|A|=n$ and $|B|=k$. Then each vertex in $A$ has degree $k$.

We prove that $mc_k(K_{k,n})\leq\lfloor\frac{k}{2}\rfloor$ by contradiction.
Suppose $mc_k(K_{k,n})=t\geq\lfloor\frac{k}{2}\rfloor+1$. For a vertex $u$ of $A$, let $d_{G_i}(u)=r_i$.
Then $\sum_{i\in [t]}r_i=k$ and each $r_i\geq1$. Because every two vertices of $A$ are connected by $k$ edge-disjoint monochromatic paths,
and the degree of every vertex in $A$ is $k$, we have that for each $u\in A$, $d_{G_i}(u)=r_i$. Because $t\geq\lfloor\frac{k}{2}\rfloor+1$,
there is a color $i$ such that $d_{G_i}(u)=1$, i.e., all vertices of $A$ are leaves of $G_i$. Because $K_{k,n}$ is a bipartite graph with
bipartition $A$ and $B$, $G_i$ is a perfect matching if $n=k$, and $G_i$ is the union of $k$ stars if $n>k$, both of which contradict
that $G_i$ is a connected spanning subgraph of $G$. Therefore, $mc_k(K_{k,n})\leq\lfloor\frac{k}{2}\rfloor$.
\end{proof}

\begin{corollary}
Conjecture \ref{coj} is true for $G=K_{k,n}$, where $k$ is even and $ n\geq k\geq 4$; it is also true for $G=K_{3,n}$, where $k=3\leq n$.
\end{corollary}
\begin{proof}
If $k=2l$ is even, then we prove that $mc_k(K_{k,n})=\lfloor\frac{k}{2}\rfloor=l$. Actually, we only need to construct an $MC_k$-coloring of
$K_{k,n}$ with $l$ colors. Let $A_1$ be a subset of $A$ with $k$ vertices and $A_2=A-A_1$, and let $H$ be the subgraph of $K_{k,n}$ whose vertex set
is $A_1\cup B$. Then $H=K_{k,k}$, and by Fact \ref{bip} $H$ can be decomposed into $l$ Hamiltonian cycles $\{C_1,\cdots,C_l\}$.
Because the degree of each vertex in $A_2$ is $k=2l$, we mark each two edges incident with $v\in A_2$ with $i$, $1\leq i\leq l$.
Let $E_i$ be the edge set with mark $i$, and let $G_i=C_i\cup E_i$. It is obvious that $G_i$ is a $2$-edge-connected spanning graph of $K_{k,n}$.
We color every edge of $G_i$ by $i$, and then we find an $MC_k$-coloring of $K_{k,n}$ with $l$ colors.

Because $K_{3,n}$ is a minimal $3$-edge-connected graph for $n\geq 3$, and an $MC_3$-coloring of $K_{3,n}$ assigns color $1$ to all its edges,
we have $mc_3(K_{3,n})\geq1$. By Theorem \ref{kkn}, $mc_3(K_{3,n})\leq1$, and thus $mc_3(K_{3,n})=1$.
\end{proof}

If $k\leq \Psi(G)$, then $G$ is $k$-edge-connected. By Theorem \ref{N-T}, there are $k$ edge-disjoint spanning trees $T_1,\cdots,T_k$ of $G$
and we color $E(G)$ such that each $T_i$ is colored by $i$. Then any two vertices $u,v$ are connected by at least $k$ monochromatic $uv$-paths with different colors.
So, we have the following result.
\begin{corollary} \label{3-5}
For a graph $G$ with $\Psi(G)\geq k\geq2$, $mc_k(G)\geq e(G)-k(n-2)$.
\end{corollary}

\section{Results for uniformly monochromatic
\newline $k$-edge-connection number}

The monochromatic $k$-edge-connected graph allows $k$ edge-disjoint monochromatic paths between any two vertices of the graph.
In this section, we generalize the concept of monochromatic $k$-edge-connection to uniformly monochromatic  $k$-edge-connection, and get some results.

An edge-colored $k$-edge-connected graph $G$ is
{\em uniformly monochromatic $k$-edge-connec\\ted} if every two distinct vertices are connected by
at least $k$ edge-disjoint monochromatic paths of $G$ such that all these $k$ paths have the same color. Note that for different pairs of vertices
the paths may have different colors. An edge-coloring $\Gamma$ of $G$ is a {\em uniformly monochromatic $k$-edge-connection coloring ($UMC_k$-coloring) } if it
makes $G$ uniformly monochromatically $k$-edge-connected. The {\em uniformly monochromatic $k$-edge-connection number}, denoted by $umc_k(G)$, of a $k$-edge-connected graph
$G$ is the maximum number of colors that are needed in order to make $G$ uniformly monochromatic $k$-edge-connected. An {\em extremal $UMC_k$-coloring}
of $G$ is an $UMC_k$-coloring that uses $umc_k(G)$ colors.
 We call an extremal $UMC_k$-coloring  a {\em good $UMC_k$-coloring} of $G$ if the coloring has the maximum number of trivial edges.
 A uniformly monochromatic $k$-edge-connected graph is also a monochromatic connected graph when $k=1$.

\begin{theorem} \label{umc1}
Let $G$ be a $k$-edge-connected graph with $k\geq 2$. Then $umc_k(G)=e(G)-e(H)+1$, where $H$ is a minimum $k$-edge-connected spanning subgraph of $G$.
\end{theorem}
We prove the theorem below. For convenience, we abbreviate ``monochromatic $uv$-path" as ``$uv$-path".
Let $\Gamma$ be a good $UMC_k$-coloring of $G$. Then, suppose that
the number of non-trivial colors of $\Gamma$ is $t$ and denote the set of them by $[t]$.
Let $G_i$ be the subgraph of $G$ induced by the edges with a non-trivial color $i$, $1\leq i\leq t$. Let $G'=\bigcup_{i\in [t]}G_i$.

\begin{claim} \label{2-1}
Each $G_i$ is $k$-edge-connected.
\end{claim}
\begin{proof}
Let $\pi_i$ denote the set of pairs $(u,v)$ such that there are at least $k$ edge-disjoint $uv$-paths colored by $i \in [t]$. Therefore, any vertex pair $(u,v)$ belongs to some $\pi_i$.

We first prove it by contradiction that each $G_i$ is $k$-edge-connected.

Suppose that $G_i$ is not a $k$-edge-connected graph. Then there exists a bond $C(G_i)$ with $|C(G_i)|\leq k-1$, and $G_i\backslash C(G_i)$ has two
components $M_1$ and $M_2$. Let $e=vu$ be an edge of $C(G_i)$, $u\in V(M_1)$, $v\in V(M_2)$. Then there are at most $|C(G)|\leq k-1$ edge-disjoint
paths in $G_i$ between $u,v$. Therefore there exists a $j\neq i$ of $[t]$ such that there are at least $k$ edge-disjoint $uv$-paths of $G_j$.

Recolor edges of $G_i-e$ with $j$ and keep the color of $e$, and denote the new coloring of $G$ by $\Gamma'$.

Because any non-trivial color $r\neq i$ is not changed. So, under $\Gamma'$, any pair $(x,y)\in \pi_r$ also have at least $k$ edge-disjoint $xy$-paths colored $r$.
For any pair $(x,y)=\pi_i$, if any $k$ edge-disjoint $xy$-paths (Note that $P_1,\cdots,P_k$) of $G_i$ under $\Gamma$ do not contain $e$. Then these $k$ edge-disjoint
$xy$-paths are retained. Otherwise, there is a path (Note that $P_1$) contains $e$. We choose a path $P$ of $G_j$ whose terminals are $u,v$. Then $T=(P_1\backslash e)\cup P$ is a trail between $x,y$
and $E(T)\cap \bigcup_{l\neq1}E(P_l)=\emptyset$. Let $P'$ be a $xy$-path of $T$. Then $P',P_2,\cdots,P_k$ are $k$ edge-disjoint $xy$-paths colored by $j$ (under $\Gamma'$).
Therefore, $\Gamma'$ is still an extremal $UMC_k$-coloring of $G$, but then $e$ becomes to a trivial edge, which contradicts that $\Gamma$ is good. So, each $G_i$ is $k$-edge-connected.

\end{proof}
By Claim \ref{2-1}, because $k\geq 2$, we have $e(G_i)\geq|G_i|\geq3$. Denote $G_x=\bigcup_{x\in V(G_i)}G_i$, $F_x=G'-G_x$.
\begin{claim} \label{2-2}
Each $G_x$ is a  $k$-edge-connected  spanning subgraph of $G$. Furthermore, $F_x=\emptyset$.
\end{claim}
\begin{proof}
If there is an $x\in V(G)$ such that $G_x$ is not a spanning subgraph of $G$, then there is a vertex $y\in V(G)\backslash V(G_x)$.
Because $G$ is a simple graph and $k\geq 2$, any two vertices are connected by at least one non-trivial path. It is obvious that
there are no non-trivial $xy$-path, a contradiction. Therefore, $G_x$ is a spanning subgraph of $G$.

Because each $G_i$ is $k$-edge-connected, $G_x$ is also $k$-edge-connected. Therefore, each $G_x$ is a  $k$-edge-connected spanning subgraph of $G$.

Now we prove that $F_x=\emptyset$. Otherwise, if $F_x\neq\emptyset$, then there is a $G_j\subseteq F_x$ and $|G_j|\geq 3$. Suppose that
$s$ is the minimum number such that $V(G_j)\subseteq\bigcup_{r\in [s]}G_{i_r}$, where $G_{i_1},\cdots,G_{i_s}$ are contained in $G_x$. Then, $s\leq |G_j|$. Because $k\geq2$, we have $e(G_j)\geq |G_j|\geq s$.
We have obtained a new coloring $\Gamma'$ from $\Gamma$ by recoloring each $G_{i_1},\cdots,G_{i_s}$ by $i_1$ and recoloring each edge of $G_j$ by different new colors.
Because $G^*=\bigcup_{r\in [s]}G_{i_r}$ is $k$-edge-connected graph, each pair $(a,b)$ with $(a,b)\in\{\pi_{i_1},\cdots,\pi_{i_s},\pi_j\}$ has $k$-edge-disjoint $ab$-paths colored $i_1$ under $\Gamma'$.
It is easy to check that $\Gamma'$ is a $UMC_k$-coloring. Then, the number of colors is not decreased, but the number of trivial colors is increased by at least $e(G_j)\geq3$, which contradicts
that $\Gamma$ is good. So, $F_x=\emptyset$.
\end{proof}

\begin{claim} \label{3-6}
$t=1$ and $G_1$ is a minimum $k$-edge-connected spanning subgraph of $G$.
\end{claim}
\begin{proof}
Suppose $t\geq2$. Then $V(G_1)\backslash V(G_2)\neq\emptyset$. Otherwise, if $V(G_1)\subseteq V(G_2)$, then $(u,v)\in \pi_2$ when $(u,v)\in \pi_1$.
We can recolor all edges of $G_1$ by fresh colors, and then the new coloring is also a $UMC_k$-coloring of $G$ but the number of colors is increased,
which contradicts that $\Gamma$ is extremal. So, $V(G_1)\backslash V(G_2)\neq\emptyset$, and there is a vertex $a\in V(G_1)\backslash V(G_2)$, i.e.,
$G_2\nsubseteq G_a$, $G_2\subseteq F_a$. By Claim \ref{2-2}, we have $F_a=\emptyset$, a contradiction. Therefore, $t=1$, and thus $G_1=G_a$ is a spanning subgraph of $G$.

In fact, $G_1$ is a minimum $k$-edge-connected spanning subgraph of $G$; otherwise, there exists a minimum $k$-edge-connected spanning subgraph $H$ of $G$
such that $e(H)<e(G_1)$. Coloring each edge of $H$ by $1$ and coloring the other edges by some different new colors. Then the coloring is a $UMC_k$-coloring of $G$
with more colors, which contradicts that $\Gamma$ is extremal.
\end{proof}

{\bf Proof of Theorem \ref{umc1}}:
We can prove Theorem \ref{umc1} directly by Claim \ref{3-6}.
$\blacksquare$

Because any $k$-edge-connected graph $G$ has the minimum degree $\delta(G)\geq k$, by Theorem \ref{Mader} we have that $\frac{1}{2}kn\leq e(H)\leq k(n-1)$,
where $H$ is a minimum $k$-edge-connected spanning subgraph of $G$.
\begin{corollary}
For a $k$-edge-connected graph $G$ with $k\geq2$, $e(G)-k(n-1)+1\leq umc_k(G)\leq e(G)-\frac{1}{2}kn+1$.
\end{corollary}

By definition, a $k$-edge-connected graph $G$ satisfies that $umc_k(G)\leq mc_k(G)$. Therefore, $mc_k(G)\geq e(G)-e(H)+1$,
where $H$ is a  $k$-edge-connected spanning subgraph of $G$. By this theorem, we also get a result:
A graph contains a Hamiltonian cycle if and only if $umc_2(G)=e(G)-n+1$.


\begin{thebibliography}{99}

\bibitem {B}
J.A. Bondy, U.S.R. Murty, Graph Theory, GTM 244, Springer,
2008.

\bibitem{CLW} Q. Cai, X. Li, D. Wu, Erd\H{o}s-Gallai-type results for colorful monochromatic connectivity
of a graph, J. Comb. Optim. 33(1)(2017), 123-131.

\bibitem{CLW} Q. Cai, X. Li, D. Wu, Some extremal results on the colorful monochromatic vertex-connectivity
of a graph, J. Comb. Optim. 35(2018), 1300--1311.

\bibitem{CY} Y. Caro, R. Yuster, Colorful monochromatic connectivity, Discrete Math. 311(2011),
1786-1792.

\bibitem{MGM} D. Gonzlez-Moreno, M. Guevara, J.J. Montellano-Ballesteros, Monochromatic connecting
colorings in strongly connected oriented graphs, Discrete Math. 340(4)(2017), 578-584.

\bibitem{GLQZ} R. Gu, X. Li, Z. Qin, Y. Zhao, More on the colorful monochromatic connectivity, Bull.
Malays. Math. Sci. Soc. 40(4)(2017), 1769-1779.

\bibitem{JLZ1} H. Jiang, X. Li, Y. Zhang, Total monochromatic connection of graphs, Discrete Math.
340(2017), 175-180.

\bibitem{JLZ2} H. Jiang, X. Li, Y. Zhang, More on total monochromatic connection of graphs, Ars
Combin. 136(2018), 263--275.

\bibitem{JLZ3} H. Jiang, X. Li, Y. Zhang, Erd\H{o}s-Gallai-type results for total monochromatic connection
of graphs, Discuss. Math. Graph Theory, in press.

\bibitem{JLW} Z. Jin, X. Li, K. Wang, The monochromatic connectivity of some graphs, submitted, 2016.

\bibitem{LW1} X. Li, D. Wu, The (vertex-)monochromatic index of a graph, J. Comb. Optim. 33(2017),
1443-1453.

\bibitem{LW2} X. Li D. Wu, A survey on monochromatic connections of graphs, Theory \& Appl. Graphs 0(1)(2018), Art.4.

\bibitem{M} W. Mader, A reduction methond for edge-connectivity in graphs, Adv. Graph Theory 3(1978), 145-164.

\bibitem{MWYY} Y. Mao, Z. Wang, F. Yanling, C. Ye, Monochromatic connectivity and graph products,
Discrete Math, Algorithm. Appl. 8(01)(2016), 1650011.


\bibitem{NW} C. St. J.A. Nash-Williams, Edge-disjoint spanning trees of
finite graphs, J. Lond. Math. Soc. 36(1961), 445--450.

\bibitem{T} W. T. Tutte, On the problem of decomposing a graph into $n$
connected factors, J. Lond. Math. Soc. 36(1961), 221--230.

\end{thebibliography}
\end{document}